\documentclass[preprint]{elsarticle}

\usepackage{amsmath,amsthm,amssymb}

\usepackage{graphics}
\usepackage{hyperref}
\usepackage[usenames, dvipsnames]{xcolor}

\theoremstyle{plain}
\newtheorem{theorem}{Theorem}[section]
\newtheorem*{theorem*}{Theorem}
\newtheorem{lemma}[theorem]{Lemma}
\newtheorem{proposition}[theorem]{Proposition}
\newtheorem*{proposition*}{Proposition}
\newtheorem{corollary}[theorem]{Corollary}
\newtheorem*{corollary*}{Corollary}

\theoremstyle{definition}
\newtheorem{remark}[theorem]{Remark}

\numberwithin{equation}{section}

\renewcommand{\Im}{\operatorname{Im}}
\renewcommand{\Re}{\operatorname{Re}}

\DeclareMathOperator{\SL}{SL}
\DeclareMathOperator{\GL}{GL}

\begin{document}

\begin{frontmatter}

\title{Sign Changes of Coefficients and Sums of Coefficients of $L$-functions}

\author[colby]{Thomas A. Hulse\fnref{tomgrant}}
\ead{tahulse@colby.edu}

\author[maine]{Chan Ieong Kuan}
\ead{chan.i.kuan@maine.edu}

\author[brown]{David Lowry-Duda\corref{cor}\fnref{davidgrant}}
\ead{djlowry@math.brown.edu}

\author[brown]{Alexander Walker}
\ead{alexander\_walker@brown.edu}

\address[colby]{Colby College, 5830 Mayflower Hill Dr, Waterville, ME 04901}
\address[maine]{University of Maine, 5752 Neville Hall Room 333, Orono, ME 04469}
\address[brown]{Brown University, 151 Thayer Street, Box 1917, Providence, RI 02912}

\fntext[tomgrant]{Partially supported by a Coleman Postdoctoral Fellowship at Queen's University.}
\fntext[davidgrant]{Supported by the National Science Foundation Graduate Research Fellowship Program under Grant No. DGE 0228243.}

\cortext[cor]{Corresponding Author}

\begin{abstract}
  We extend the axiomatization for detecting and quantifying sign changes of Meher and Murty~\cite{MM} to sequences of complex numbers.
We further generalize this result when the sequence is comprised of the coefficients of an $L$-function.
As immediate applications, we prove that there are sign changes in intervals within sequences of coefficients of $\GL(2)$ holomorphic cusp forms, $\GL(2)$ Maass forms, and $\GL(3)$ Maass forms.
Building on~\cite{hkldw1, hkldw2}, we prove that there are sign changes in intervals within sequences of partial sums of coefficients of $\GL(2)$ holomorphic cusp forms and Maass forms.
\end{abstract}


\end{frontmatter}


\section{Introduction}

Let $L(s,f)$ be an $L$-function with associated Dirichlet series
\begin{equation*}
  L(s,f) = \sum_{n \geq 1} \frac{A_f(n)}{n^s},
\end{equation*}
for $\Re(s)$ sufficiently large and complex coefficients $\{A_f(n)\}_{n \in \mathbb{N}}$.
A complete definition of $L$-functions can be found in~\cite[Chapter 5]{iwaniec2004analytic}, but we consider only those properties associated with $L(s,f)$ having a functional equation $s \mapsto 1-s$ and sometimes with the coefficients satisfying the Ramanujan-Petersson conjecture, $A_f(n)=O(n^{\epsilon})$ for all $\epsilon>0$.
Among many other results proven in great generality, Chandrasekharan and Narasimhan~\cite{CN1, CN2} demonstrated upper bounds on partial sums of coefficients of automorphic $L$-functions,
\begin{equation} \label{sf1}
 S_f(n) := \sum_{m \leq n} A_f(m),
\end{equation}
 and gave conditions indicating when these partial sums change sign infinitely often.
Further, they identified conditions guaranteeing the existence of an $\alpha > 0$ such that, if $\Re A_f(n) \neq 0$ for some $n$, then $\Re S_f(X) = \Omega_{\pm} (X^\alpha)$, and similarly for $\Im A_f(n)$.
When applied to the coefficients $\{A_f(n)\}$ of holomorphic cusp forms of weight $k \in \frac{1}{2}\mathbb{Z}$, their method gives a lower bound of $\alpha= \frac{1}{4}$.
Hafner and Ivi\'{c}~\cite{HI} gave improved upper bounds for $S_f(n)$ and improved sign change results in the case of full-integral weight forms of level one.
While their techniques do extend to higher level, they do not easily extend to more general modular forms or half-integer weight forms.
In~\cite{HBT}, Heath-Brown and Tsang were able to show that $\Delta(n)$ and $P(n)$, denoting the respective error terms in Dirichlet's divisor problem and Gauss's circle problem, change sign for $n \in [X,X+c\sqrt{X}]$ for some constant $c$ and large $X$.
However they also showed that there are always intervals of length $\sqrt{X}\log^{-5}(X)$ inside intervals of the form $[X,2X]$ where sign changes of $\Delta(X)$ and $P(X)$ do not occur.

Sign change results for sums of coefficients imply sign change results for individual coefficients, but far more can be said about this by studying these problems directly.
Many specific cases were investigated in works such as \cite{Knopp03}, \cite{Kohnen07}, \cite{Kohnen08}, \cite{Kohnen}, \cite{HKKL}, \cite{Choie2013}, \cite{Inam2013}, \cite{Narasimha13}, \cite{Inam2}, \cite{Inam1} and \cite{MM2}. 
General sign change results for Dirichlet series with holomorphic continuation were obtained by Pribitkin~\cite{P1} and were later applied specifically to cases of automorphic forms~\cite{P2}.

More recently, Meher and Murty~\cite{MM} have shown that for half-integral weight forms with real coefficients, sequences of coefficients, $\{A_f(n)\}_{n \in \mathbb{N}}$, have at least one sign change for $n \in (X, X + X^\delta]$ for $\delta > \frac{43}{70}$ and sufficiently large $X$, giving quantitative results on the number of sign changes.
They further outline a general set of growth conditions on partial sums guaranteeing quantitative results on the number of sign changes of real-valued coefficients in a Dirichlet series.
They note that with the assumption of the Ramanujan-Petersson conjecture, they can take $\delta > \frac{3}{5}$, but actually the bound on individual coefficients due to Bykovskii~\cite{Bykovskii98}, $A_f(n)=O(n^{\frac{3}{16}+\epsilon})$ is enough to remove this assumption.
Indeed, they achieve this later in~\cite{MM2} in the case of real coefficients of $L$-functions of automorphic irreducible self-dual cuspidal representations  of $\GL_2(\mathbb{A}_{\mathbb{Q}})$.
In another work by Meher and others~\cite{MJP}, sign change results are obtained immediately from the analytic properties of automorphic $L$-functions.

In Section~\ref{sec:sign_change_statements} of this paper, we extend the Meher-Murty axiomatization guaranteeing quantitative results on sign changes to sequences of complex coefficients.
We also present a new axiomatization that gives intervals of sign changes directly from the growth of associated $L$-functions on vertical strips.
In Section~\ref{sec:individual_gl2} we continue following the path laid in~\cite{MM} and~\cite{MM2} and obtain general sign change results for coefficients of $\GL(2)$ cusp forms of full or half-integral weight that are not necessarily real.
In Section~\ref{sec:individual_gl3}, we prove generalized sign change results for the coefficients of $\GL(3)$ Maass forms.
In particular, if $\{A(m,n)\}_{m,n \in \mathbb{Z}}$ denotes the coefficients of a Maass Hecke cusp form on $\SL_3(\mathbb{Z})$, then we prove that $\{ \Re A(1,n) \}_{n \in \mathbb{N}}$ has at least one sign change for $n \in (X, X + X^{\frac{21}{23} + \epsilon}]$ for all $X$ large enough.
The method seems generalizable to the $\GL(m)$ case, and we discuss this and compare it to more results in the recent work of Meher and Murty in~\cite{MM2} and what we can obtain from the Ramanujan-Petersson Conjecture and the Generalized Lindel\"{o}f Hypothesis.

Building off of the notation in \eqref{sf1}, we let $S_f^\nu$ denote the partial sums of $\nu$-normalized coefficients associated to some object $f$,
\begin{equation} \label{sf2}
  S_f^\nu(n) := \sum_{m \leq n} \frac{B_f(m)}{m^\nu},
\end{equation}
where  $\nu \in \mathbb{R}$ and $B_f(m)$ may not necessarily be normalized to satisfy a Ramanujan-Petersson conjecture on average.
In Section~\ref{sec:sums_gl2} we investigate sign changes of $\{S_f^\nu(n)\}_{n \in \mathbb{N}}$ for $\GL(2)$ Maass forms and both full and half-integral weight holomorphic cusp forms.
This builds on the investigation in~\cite{hkldw1, hkldw2}, in which the authors initiated an investigation into the analytic properties of the three Dirichlet series,
\begin{align*}
  L(s, S_f) &:= \sum_{n \geq 1} \frac{S_f(n)}{n^s}, \\
  L(s, S_f\times S_f) := \sum_{n \geq 1} \frac{S_f(n)\overline{S_f(n)}}{n^s}, &\qquad L(s, S_f\times \overline{S_f}) := \sum_{n \geq 1} \frac{S_f(n) S_f(n) }{n^s},
\end{align*}
when $f$ was a holomorphic cusp form.
There the authors showed that each of the above series has a meromorphic continuation to $s \in \mathbb{C}$ and used this to derive improved short-interval averages for $S_f(n)$ and $S_f(n)^2$.
The bound on our sign-change interval naturally increases as the normalization $\nu$ is taken to be larger.
We show that there are sign changes in intervals even when $S_f^\nu$ is ``overnormalized'' in the sense that the size of the individual coefficients $B_f(n)n^{-\nu}$ is decreasing in $n$.
This hints at some additional structure in the size and sign changes in $\{B_f(n)\}_{n \in \mathbb{N}}$.

Beyond the examples we treat here, we expect the generalized axiomatization could be applied to various other families of sequences, as was done with the coefficients of Hilbert modular forms by Meher and Tanabe~\cite{MT}.
Furthermore, as the symmetric square $L$-function for $f$ is the key to obtaining a sign change result for the real and imaginary parts of complex numbers, it is plausible that our extension of the Meher-Murty axiomatization could be further generalized by including information about $L(s,\mbox{Sym}^nf)$, the $n$-th symmetric power $L$-function of $f$.
This would not be surprising, as the Sato-Tate conjecture for the probabilistic distribution of coefficients similarly follows from the conjectured analytic properties of $L(s,\mbox{Sym}^nf)$, which are in turn expected to follow from the Langlands program~\cite[Chapter 21]{iwaniec2004analytic}.

\subsection*{Acknowledgements}

We would like to thank Jeff Hoffstein at Brown University, who by asking a question at a talk given by Winfried Kohnen at Jeff's 61st birthday conference prompted this series of papers.
We would also like to thank M.\ Ram Murty, Jaban Meher, and Naomi Tanabe who informed the authors of this sign change axiomatization via the Queen's University Number Theory Seminar, and for some helpful conversations.

\section{Criteria for Generalized Sign Changes in Intervals}
\label{sec:sign_change_statements}

Meher and Murty~\cite{MM} identified a set of conditions guaranteeing sign changes in a sequence $\{a(n)\}_{n \in \mathbb{N}}$ of real coefficients.
Heuristically, it's clear that $a(n)$ should change signs many times if $\sum_{n \leq X} a(n)$ is small while $\sum_{n \leq X} \lvert a(n) \rvert^2$ is large.
We extend their results to sequences of complex numbers by showing that terms escape from ``wedges'' infinitely often, where we define a wedge to be the portion
\begin{equation} \label{wedge}
\mathcal{W}(\theta_1, \theta_2) = \{re^{i\theta} : r > 0, \theta \in [\theta_1, \theta_2]\}
\end{equation}
of the complex plane where $0 \leq \theta_2 - \theta_1 < \pi$.

Further, if $\sum_{n \leq X} a(n)^2$ is also small, we show that $\{\Re a(n)\}_{n \in \mathbb{N}}$ and $\{\Im a(n)\}_{n \in \mathbb{N}}$ each have many sign changes.

\begin{theorem}\label{theorem:main_sign_changes}
  Let $\{a(n)\}_{n \in \mathbb{N}}$ be a sequence of complex numbers satisfying
  \begin{enumerate}
    \item $a(n) = O(n^\alpha)$
    \item $\sum_{n \leq X} a(n) \ll X^\beta$
    \item\label{item:sum_abs_squares} $\sum_{n \leq X} \lvert a(n) \rvert^2 = c_{1} X^{\gamma_{1}} + O(X^{\eta_{1}})$,
  \end{enumerate}
  where $\alpha, \beta, c_{1}, \gamma_{1}, \eta_{1}$ are positive real constants.
   Then for any $r$ satisfying
   \begin{equation}\label{rbound1}
   \max(\alpha + \beta, \eta_{1}) - (\gamma_{1} - 1) < r < 1,
   \end{equation}
  the sequence $\{a(n)\}_{n \in \mathbb{N}}$ has at least one term outside of $\mathcal{W}(\theta_1, \theta_2)$ for some $n \in (X, X + X^r]$ for all $X \gg 1$.\\
 \\
Now also suppose that $\{a(n)\}_{n \in \mathbb{N}}$ satisfies the additional condition
  \begin{enumerate}
      \setcounter{enumi}{3}
    \item\label{item:sum_squares} $\sum_{n \leq X} a(n)^2 = c_{2} X^{\gamma_{2}} + O(X^{\eta_{2}})$,
  \end{enumerate}
  where $\gamma_{2}$ and $\eta_{2}$ are positive real constants, and $c_{2}$ is possibly complex, and that $\alpha + \beta < \max(\gamma_1, \gamma_2)$.
  Let $r$ be such that
   \begin{equation}\label{rbound2}
       \max(\alpha + \beta, \eta_{2}, \eta_{1}) - (\max(\gamma_{2}, \gamma_{1}) - 1) < r < 1.
  \end{equation}
  If $\Re(c_{2})X^{\gamma_{2}} \neq -c_{1} X^{\gamma_{1}}$ (resp.\ if $\Re(c_{2})X^{\gamma_{2}} \neq c_{1} X^{\gamma_{1}}$) then the sequence $\{\Re(a(n))\}_{n \in \mathbb{N}}$ (resp.\ $\{\Im(a(n))\}_{n \in \mathbb{N}}$) has at least one sign change in the interval $(X, X+X^r]$ for $X \gg 1$.
\end{theorem}

\begin{remark}
  Note in particular that if the $a(n)$ were all real, this is equivalent to there being at least one sign change in the interval $(X, X + X^r]$.
  The wedge result is sharp in that it cannot be improved to half-plane wedges or larger without further assumptions.
  For instance, one can consider the sequence given by $a(n) = (-1)^n + i/n^2$, which satisfies conditions (1)-(3) but never escapes the half-plane.

  When condition (4) and the subsequent hypotheses also hold, the sign-change in the real and imaginary parts is a strictly stronger result than escaping a wedge.
  Indeed, since we can replace $a(n)$ with $e^{i\phi}a(n)$ in the hypothesis of the above theorem for any $\phi \in \mathbb{R}$, it follows that we obtain sign changes for $\Re(e^{i\phi}a(n))$ and $\Im(e^{i\phi}a(n))$, which means the coefficients $a(n)$ escape any half-plane.
\end{remark}

To prove the first part of Theorem~\ref{theorem:main_sign_changes}, we will need the following lemma.

\begin{lemma}\label{lem:coslemma} Suppose that the sequence of complex numbers $\{ a(n) \}_{n=1}^{N}$ lies inside the wedge
  \begin{equation*}
    \mathcal{W}(\theta_1,\theta_2) = \left\{re^{i\theta} \ |  \ r>0, \  \theta \in [\theta_1,\theta_2] \right\}
  \end{equation*}
  where $0 \leq \theta_2-\theta_1 <\pi$.
  Then
  \begin{equation*}
    \left| \sum_{n=1}^N a(n) \right| \geq \cos\left(\frac{\theta_2-\theta_1}{2}\right) \sum_{n=1}^N |a(n)|.
  \end{equation*}
\end{lemma}
\begin{proof}
  Since $e^{i\phi}\mathcal{W}(\theta_1,\theta_2) = \mathcal{W}(\theta_1+\phi,\theta_2+\phi)$,	we suppose without loss of generality that $\{a(n)\}$ lies in the wedge $\mathcal{W}(-\phi,\phi)$, where $\phi=\frac{\theta_2-\theta_1}{2} \in [0,\frac{\pi}{2})$.
  Observe that $a(n) = e^{i\theta}|a(n)|$ for some $\theta \in [-\phi,\phi]$ and so
  \begin{equation*}
    \Re(a(n)) =\cos(\theta)|a(n)| \geq \cos(\phi)|a(n)|.
  \end{equation*}
  Since for any $z \in \mathbb{C}$, $|z| \geq |\Re(z)|$, and since $\Re a(n) \geq 0$, we have that
  \begin{equation*}
    \left| \sum_{n=1}^N a(n) \right| \geq \sum_{n=1}^N \Re(a(n)) \geq \cos(\phi) \sum_{n=1}^N |a(n)|,
  \end{equation*}
  which completes the proof of the lemma.
\end{proof}

\begin{proof}[Proof of Theorem~\ref{theorem:main_sign_changes}]
  We first show the statement about wedges.
  Fix $\theta_1,\theta_2$ such that $0\leq \theta_2 - \theta_1 < \pi$.
  Assume the conditions (1)-(3) of $\{a(n)\}$ and suppose for the sake of obtaining a contradiction that the claim is false, so that we have an $r$ satisfying \eqref{rbound1}, but $a(n) \in \mathcal{W}(\theta_1,\theta_2)$ for all $n \in (X, X+X^r]$.
  Then by condition (1), Lemma~\ref{lem:coslemma}, and the fact that $r<1$, we have
  \begin{equation*}
    \sum_{X < n < X+X^r} |a(n)|^2 \ll X^\alpha \sum_{X < n < X + X^r} |a(n)| \ll \frac{X^{\alpha + \beta}}{\cos(\tfrac{\theta_2-\theta_1}{2})}.
  \end{equation*}
  As $\theta_2 - \theta_1 < \pi$, the cosine term is a fixed positive constant.
  Since $r+\gamma_{1}-1 > \eta_{1}$,
  \begin{equation*}
    \sum_{X < n < X + X^r} |a(n)|^2  = c_1(X+X^r)^{\gamma_{1}}-c_1X^{\gamma_{1}} +O(X^{\eta_1}) \gg cX^{r+\gamma_{1}-1},
  \end{equation*}
  by the binomial theorem.
  As $X$ gets large, it follows that $\alpha+\beta \geq \gamma_{1}+r-1$, contradicting the inequality for $r$.

  Supposing now that $\{a(n)\}_{n \in \mathbb{N}}$ also satisfies condition~\eqref{item:sum_squares}, we prove the second half of the theorem.
  Note that $a(n)^2 = \Re a(n)^2 - \Im a(n)^2 + 2i\Re a(n)\Im a(n)$ and $\lvert a(n) \rvert^2 = \Re a(n)^2 + \Im a(n)^2$.
  We can therefore rewrite conditions~\eqref{item:sum_abs_squares} and~\eqref{item:sum_squares} as
  \begin{align*}
  &    \sum_{n \leq X} \Re a(n)^2 + \Im a(n)^2 = c_{1} X^{\gamma_{1}} + O(X^{\eta_{1}}), \\
  & \sum_{n \leq X} \Re a(n)^2 - \Im a(n)^2 + 2i\Re a(n)\Im a(n) = c_{2} X^{\gamma_{2}} + O(X^{\eta_{2}}),
  \end{align*}
  respectively.
  Take real parts and add these together to get
  \begin{equation*}
    2\sum_{n \leq X} \Re(a(n))^2 = \Re(c_{2})X^{\gamma_{2}} + c_{1} X^{\gamma_{1}} + O(X^{\eta_{2}}+X^{\eta_{1}}).
  \end{equation*}
Suppose without loss of generality that $\Re a(n) \geq 0$ for $n \in (X, X+X^r]$ for some $r$ satisfying \eqref{rbound2}.
  Then on one hand,
  \begin{equation*}
    2\sum_{X < n < X + X^r} (\Re a(n))^2 \ll X^\alpha \sum_{X < n < X + X^r} \Re a(n) \ll X^{\alpha + \beta}.
  \end{equation*}
  On the other hand,
  \begin{align*}
    2\sum_{X < n < X + X^r} (\Re a(n))^2 &= \Re(c_{2})\left((X + X^r)^{\gamma_{2}} - X^{\gamma_{2}} \right) +O(X^{\eta_{2}}) \\[-2ex]
    &\quad+ c_{1}\left( (X+X^r)^{\gamma_{1}} - X^{\gamma_{1}}\right) + O(X^{\eta_{1}}) \\
    &= \Re(c_{2})X^{\gamma_{2} + r - 1} + O(X^{\eta_{2}} + X^{\gamma_{2} + 2(r-1)}) \\
    &\quad+ c_{1} X^{\gamma_{1} + r - 1} + O(X^{\eta_{1}} + X^{\gamma_{1} + 2(r-1)}).
  \end{align*}
 Unless $\Re(c_2)X^{\gamma_2}+c_1X^{\gamma_1}=0$, it follows that
  \begin{equation*}
    X^{\gamma_{2} + r - 1} + X^{\gamma_{1} + r - 1} = O(X^{\alpha + \beta} + X^{\eta_{2}} + X^{\eta_{1}}).
  \end{equation*}
 But $\max(\gamma_{2}, \gamma_{1}) + r - 1 > \max(\alpha + \beta, \eta_{2}, \eta_{1})$, so we obtain a contradiction.
 A similar argument holds for the imaginary parts, using $\Re(c_{2})X^{\gamma_{2}} \neq c_{1} X^{\gamma_{1}}$.
\end{proof}

In practice, we often analyze Dirichlet series
and perform a series of cutoff integral transforms to obtain the growth conditions necessary for Theorem~\ref{theorem:main_sign_changes} to hold.
It is natural to want to translate directly from the analytic behavior to sign changes in intervals, and in Theorem~\ref{theorem:signchanges_analytic} we obtain a result to this end.

\begin{theorem}\label{theorem:signchanges_analytic}
  Let $\{a(n)\}_{n \in \mathbb{N}}$ be a sequence of complex numbers satisfying the following:
  \begin{enumerate}
    \item $a(n) = O(n^\alpha)$
    \item The continuation of the Dirichlet series $\sum_{n \geq 1} a(n) n^{-s}$ is entire, and at $\Re s = \sigma_0$ is $O( (1 + \lvert \Im s \rvert)^\beta)$
    \item The continuation of the Dirichlet series $\sum_{n \geq 1} \lvert a(n) \rvert^2 n^{-s}$ is holomorphic for $\Re s \geq \sigma_1$ except for a simple pole at $s = \gamma_{1}$ with residue $R_1$. Further, at $\Re s = \sigma_1$, this Dirichlet series is $O((1 + \lvert \Im s \rvert)^{\eta_{1}})$,
  \end{enumerate}
  in which $\alpha, \beta, \gamma_{1}, \eta_{1}$, and $R_1$ are all real constants.
  Fix a wedge $\mathcal{W}(\theta_1, \theta_2)$ as in \eqref{wedge}.
  Then the sequence $\{a(n)\}_{n \in \mathbb{N}}$ has at least one term outside $\mathcal{W}(\theta_1, \theta_2)$ for some $n \in (X, X + X^r]$ for any $r$ satisfying
  \begin{equation}\label{rbound3}
  1 +\max \left\{ \frac{\sigma_1 - \gamma_{1}}{1 + \eta_{1}}, \frac{\alpha + \sigma_0 - \gamma_{1}}{1 + \beta}\right\}< r < 1
  \end{equation} and $X$ sufficiently large.\\
\\
  Further, suppose $\{a(n)\}_{n \in \mathbb{N}}$ also satisfies:
  \begin{enumerate}
    \setcounter{enumi}{3}
    \item The Dirichlet series $\sum_{n \geq 1} a(n)^2 n^{-s}$ is holomorphic for $\Re s \geq \sigma_2$ except for a simple pole at $s = \gamma_{2}$ with residue $R_2$. Further, at $\Re s = \sigma_2$, this Dirichlet series is $O((1 + \lvert \Im s \rvert)^{\eta_{2}})$,
  \end{enumerate}
  in which $\gamma_{2},\eta_2$ are real and $R_2$ is possibly complex.
  Define $\gamma := \max(\gamma_{1}, \gamma_{2})$.
  If $\Re(R_2)X^{\gamma_{2}} \neq -R_1 X^{\gamma_{1}}$ and $\Re(R_2) X^{\gamma_{2}} \neq R_1 X^{\gamma_{1}}$, then the sequences $\{\Re a(n)\}_{n \in \mathbb{N}}$ and $\{\Im a(n) \}_{n \in \mathbb{N}}$ each have at least one sign change for some $n \in (X, X + X^r]$ for any $r$ satisfying
  \begin{equation} \label{rbound4}
  1 + \max\left\{ \frac{\sigma_1 - \gamma}{1 + \eta_{1}}, \frac{\sigma_2 - \gamma}{1 + \eta_{2}}, \frac{\alpha + \sigma_0 - \gamma}{1 + \beta} \right\}< r < 1,
  \end{equation}
   for $X \gg 1$.
\end{theorem}

\begin{proof}
  For ease of description, suppose that $a(n)$ is real.
  For $X, Y > 0$, let $v_Y(X)$ denote a smooth nonnegative cutoff function with maximum value $1$, compactly supported in the interval $[1, 1 + \frac{1}{Y}]$, and satisfying $v_Y(X) = 1$ for $1 + \frac{1}{4Y} \leq X \leq 1 + \frac{3}{4Y}$.
  Let $V_Y(s)$ denote the corresponding Mellin transform of $v_Y(x)$.
  The Mellin pair $v_Y(X)$ and $V_Y(s)$ are very similar to the pair we will discuss in greater detail in Section~\ref{sec:analytic_bounds}, except $v_Y(X)$ is supported on a smaller interval and, correspondingly, $V_Y(s)$ has no pole.
  To obtain a contradiction, suppose that $a(n) > 0$ for all $n \in (X, X + \frac{X}{Y}]$, then
  \begin{equation*}
    \sum_{n \geq 1} a(n)^2 v_Y\left(\frac{n}{X}\right) \ll X^\alpha \sum_{n \geq 1} a(n) v_Y\left(\frac{n}{X}\right) = \frac{X^\alpha}{2\pi i} \int_{(\sigma)} \sum_{n \geq 1} \frac{a(n)}{n^s} X^s V_Y(s)ds
  \end{equation*}
for sufficiently large $\sigma$.
Moving the line of integration to $\Re s = \sigma_0$ and bounding $V_Y(s)$ using integration by parts (analogous to the bound \eqref{vbound} given in Section~\ref{sec:analytic_bounds}) we get that
  \begin{equation}\label{eq:smoothed_L1}
    \sum_{n \geq 1} a(n)^2 v_Y\left(\frac{n}{X}\right) \ll Y^{\beta + \epsilon}X^{\alpha+\sigma_0}.
  \end{equation}
  Similar calculations give the smoothed asymptotic,
  \begin{equation}\label{eq:smoothed_L2}
    \sum_{n \geq 1} a(n)^2 v_Y(\frac{n}{X}) = V_Y(\gamma_{1}) R_1 X^{\gamma_{1}} + O(Y^{\eta_{1} + \epsilon}X^{\sigma_1}).
  \end{equation}
  Combining \eqref{eq:smoothed_L1} and \eqref{eq:smoothed_L2}, and noting that $V_Y(s) \asymp \frac{1}{Y}$, following the notation of~\cite{iwaniec2004analytic}, we get that
  \begin{equation*}
    \frac{X^{\gamma_{1}}}{Y} = O(Y^{\eta_{1} + \epsilon}X^{\sigma_1} + Y^{\beta + \epsilon}X^{\alpha + \sigma_0}),
  \end{equation*}
  which implies that
  \begin{equation*}
    Y \gg X^{\frac{\gamma_{1} - \sigma_1}{1 + \eta_{1}+\epsilon}} \quad \text{or} \quad Y \gg X^{\frac{\gamma_{1} - \sigma_0 - \alpha}{1 + \beta+\epsilon}}.
  \end{equation*}
So if there is no sign change in the interval $(X, X + \frac{X}{Y}]$ then
  \begin{equation*}
    \frac{X}{Y} \ll X^{1 + \max\{\frac{\sigma_1 - \gamma_{1}}{1 + \eta_{1}+\epsilon}, \frac{\alpha + \sigma_0 - \gamma_{1}}{1 + \beta+\epsilon}\}}.
  \end{equation*}
  For any $r$ satisfying~\eqref{rbound3}, we can take $\epsilon>0$ to be small enough such that the contrapositive completes the proof for the case when $a(n)$ are real, and the other cases follow as in the proof of Theorem~\ref{theorem:main_sign_changes}.
\end{proof}

\section{Individual Coefficients of $\GL(2)$ Cusp Forms}
\label{sec:individual_gl2}

As a first application, we prove general sign change results for the (possibly non-real) coefficients of cusp forms on $\GL(2)$, particularly for the coefficients of weight zero Maass forms and holomorphic cusp forms of integral and half-integral weight.
The argument in this section takes heavily from the argument for sign changes of half-integral weight cusp forms in~\cite{MM} but makes use of our generalization of their axiomatization.

For this section and subsequent sections, we will use~\cite[Theorem~4.1]{CN1} and~\cite[Theorem~1]{CN2} repeatedly in the following form.

\begin{theorem}[Chandrasekharan and Narasimhan, 1962 and 1964] \label{theorem:CN}
  Let $f$ and $g$ be objects with meromorphic Dirichlet series
  \begin{equation*}
    L(s,f) = \sum_{n \geq 1} \frac{a(n)}{n^s}, \qquad L(s,g)=\sum_{n\geq 1} \frac{b(n)}{n^s}.
  \end{equation*}
  Suppose $G(s) = Q^s\prod_{i=1}^\ell \Gamma(\alpha_i s + \beta_i)$ is a product of gamma factors with $Q>0$ and $\alpha_i > 0$.
  Define $A = \sum_{i=1}^\ell \alpha_i$.
  Let $w$ and $w'$ be numbers such that $\sum_{n \leq X} \lvert b(n) \rvert^2 \ll X^{2w - 1} \log^{w'} X$.
  Let
  \begin{equation*}
    Q(X) = \frac{1}{2\pi i} \int_{\mathcal{C}} \frac{L(s,f)}{s} X^s \ ds,
  \end{equation*}
  where $\mathcal{C}$ is a smooth closed contour enclosing all the singularities of the integrand.
  Let $q$ be the maximum of the real parts of the singularities of $L(s,f)$ and let $r$ be the maximum order of a pole of $L(s,f)$ with real part $q$.
  Suppose the functional equation
  \begin{equation}\label{eq:basic_feq}
    G(s)L(s,f) = \epsilon(f) G(\delta-s)L(\delta-s,g)
  \end{equation}
  is satisfied for some $\lvert \epsilon(f) \rvert = 1$ and $\delta >0$.
  Then we have that
  \begin{equation}
    \begin{split}\label{eq:CN_L1}
      S_f(X) =& \sum_{n\leq X} a(n) =
      Q(X) + O(X^{\frac{\delta}{2}-\frac{1}{4A}+2A(w-\frac{\delta}{2}-\frac{1}{4A})\eta + \epsilon}) \\
      & \quad + O(X^{q-\frac{1}{2A}-\eta}\log(X)^{r-1})+O\bigg( \sum_{X \leq n \leq X'} |a(n)| \bigg)
    \end{split}
  \end{equation}
  for any $\eta \geq 0$, and where $X' =X+O(X^{1-\frac{1}{2A}-\eta})$.
  If all $a(n) \geq 0$, the final $O$-error term above does not contribute.

  Suppose further that $A \geq 1$ and that $2w - \delta - \frac{1}{A} \leq 0$.
  Then
  \begin{equation*}
    \sum_{n \leq X} \lvert S_f(n) - Q(n) \rvert^2 = cX^{\delta - \frac{1}{2A} + 1} + O(X^{\delta} \log^{w' + 2} X)
  \end{equation*}
  for a constant $c$ that can be made explicit.
\end{theorem}

Now we use this theorem to deduce bounds necessary to apply Theorem~\ref{theorem:main_sign_changes} to individual coefficients of cusp forms on $\GL(2)$.

\begin{corollary} \label{cor:adj}
  Let $f$ be either a weight zero Maass form with eigenvalue $\lambda_j = \frac{1}{4} + t_j^2$ and Fourier expansion
  \begin{equation}\label{eq:maass_expansion}
    f(z) = \sum_{n \neq 0} A_f(n) \sqrt{y} K_{it_j}(2\pi i \lvert y \rvert n) e^{2\pi i n x},
  \end{equation}
  or a holomorphic cusp form of full or half integer weight with Fourier expansion
  \begin{equation}\label{eq:holomorphic_expansion}
    f(z) = \sum_{n \geq 1} A_f(n) n^{\frac{k-1}{2}} e^{2\pi i n z}.
  \end{equation}
  In any of the above cases $f$ may be of level $N$ and possibly with nontrivial nebentypus.
  Then
  \begin{equation}\label{rankin1}
    \sum_{n\leq X} |A_f(n)|^2 = c_fX+O(X^{\frac{3}{5}+\epsilon})
  \end{equation}
  for some effectively computable constant $c_f$.
\end{corollary}

\begin{proof}
  In each case, the Rankin-Selberg convolution $L$-function $L(s,f \times f)$ can be written as
  \begin{equation*}
    L(s,f\times f) =L(2s,\chi_0) \sum_{n=1}^\infty \frac{|A_f(n)|^2}{n^s} = \sum_{n \geq 1} \frac{C_f(n)}{n^s},
  \end{equation*}
for sufficiently large $\Re(s)$,  where $\chi_0$ is the trivial character mod $N$, and where we note that $C_f(n) \geq 0$.
  Further, $L(s, f\times f)$ is meromorphic with poles at $s = 1$ and possibly at $s = \frac{1}{2}$ with easily described residues $R_1$ and $R_{1/2}$, and satisfies a functional equation of the form~\eqref{eq:basic_feq} due to the functional equation of the level $N$ Eisenstein series.
  Taking $A=2, \delta=1, Q(x) = R_1 x + R_{1/2} x^{1/2}$, and $w = 1 + \epsilon$ for some small $\epsilon > 0$ and applying Theorem~\ref{theorem:CN}, it follows that


  \begin{equation*}
    \sum_{n \leq X} C_f(n) = R_1X+O(X^{\frac{3}{5}+\epsilon'}).
  \end{equation*}
  Let $\mu_N(m)=\chi_0(m) \mu(m)$, so that $L(2s,\chi_0)^{-1}= \sum_{n=1}^\infty \mu_N(n) n^{-2s}$.
  By multiplying the Dirichlet series, we get that
  \begin{equation*}
    |A_f(n)|^2 = \sum_{d^2 | n} \mu_N(d)C_f(n/d^2).
  \end{equation*}
  Then
  \begin{align*}
    \sum_{n \leq X} |A_f(n)|^2 &= \sum_{n\leq X}  \sum_{d^2 | n} \mu_N(d)C_f(n/d^2) = \sum_{d^2 \leq X} \mu_N(d) \sum_{v \leq X/d^2} C_f(v) \\
                               &= \sum_{d \leq \sqrt{X}} \mu_N(d) \left(R_1 \frac{X}{d^2}+O\left(\frac{X^{3/5}}{d^{6/5}}\right) \right).
  \end{align*}
  Since $\sum_{n=1}^\infty \mu_N(n) n^{-s}$ is convergent for $\Re s >1$  and
  \begin{equation*}
    \sum_{d^2 \leq X} \frac{\mu_N(d)}{d^2} = \frac{6}{\pi^2} \prod_{p |N} \left(1-\frac{1}{p^2} \right)^{-1} +O(1/\sqrt{X}).
  \end{equation*}
It is clear that~\eqref{rankin1} follows.
\end{proof}

\begin{corollary}\label{cor:sym}
  Let $f$ be a modular cusp form of level $N$ with non-real nebentypus $\chi$ that is either a holomorphic form as in~\eqref{eq:holomorphic_expansion} or is a weight zero Maass form as in~\eqref{eq:maass_expansion}.
  Then
  \begin{equation*}
    \sum_{n\leq X} A_f(n)^2 = O(X^{\frac{3}{5}+\epsilon})
  \end{equation*}
\end{corollary}

\begin{proof}
  The $L$-function
  \begin{equation*}
    L(s,f\times \overline{f}) = L(2s,\chi^2) \sum_{n=1}^\infty \frac{A_f(n)^2}{n^s} = \sum_{n=1}^\infty \frac{C'_f(n)}{n^s},
  \end{equation*}
  is entire.
  So one can take $q = -\infty$ and $Q(x) = O(1)$ in our application of Theorem~\ref{theorem:CN} and otherwise proceed analogously as in the proof of Corollary~\ref{cor:adj}.
\end{proof}

We collect Deligne's celebrated proof of the Ramanujan-Petersson conjecture for full-integral weight forms~\cite{Deligne}, Bykovskii's bound for half-integral weight forms~\cite{Bykovskii98}, and the Kim-Sarnak bound for Maass forms~\cite{KS} into the following theorem giving bounds on individual coefficients.

\begin{theorem}[Deligne (1974), Bykovskii (1998), Kim-Sarnak (2003)] \label{theorem:individual_bounds}
  Let $f$ be a modular cusp form of level $N$ that is either a holomorphic form as in~\eqref{eq:holomorphic_expansion} or a weight zero Maass form as in~\eqref{eq:maass_expansion}.
  Then
  \begin{equation*}
    A_f(n) \ll n^{\alpha(f) + \epsilon},
  \end{equation*}
  where
  \begin{equation*}
    \alpha(f) = \begin{cases}
      0 & \text{if } $f$ \; \text{is full-integral weight holomorphic}, \\
      \frac{3}{16} & \text{if } $f$ \; \text{is half-integral weight holomorphic}, \\
      \frac{7}{64} & \text{if } $f$ \; \text{is a Maass form}.
    \end{cases}
  \end{equation*}
\end{theorem}

Finally, we apply Theorem~\ref{theorem:CN} and reason as in~\cite{MM} to produce the following.

\begin{corollary}\label{cor:sum}
  Let $f$ be a modular cusp form of level $N$ that is either a holomorphic form as in~\eqref{eq:holomorphic_expansion} or a weight zero Maass form as in~\eqref{eq:maass_expansion}, possibly with nontrivial nebentypus $\chi$.
  Then
  \begin{equation*}
    \sum_{n \leq X} A_f(n) \ll X^{\beta(f) + \epsilon}
  \end{equation*}
  where
  \begin{equation*}
    \beta(f) = \begin{cases}
      \frac{1}{3} & \text{if } $f$ \; \text{is full-integral weight holomorphic}, \\
      \frac{19}{48} & \text{if } $f$ \; \text{is half-integral weight holomorphic}, \\
      \frac{71}{192} & \text{if } $f$ \; \text{is a Maass form}.
    \end{cases}
  \end{equation*}
\end{corollary}

\begin{proof}
  In each case, the associated $L$-function
  \begin{equation*}
    L(s,f) = \sum_{n=1}^\infty \frac{A_f(n)}{n^s}
  \end{equation*}
  is an entire function and satisfies a functional equation
  \begin{equation*}
    G(s)L(s,f)=G(1-s)L(1-s,g),
  \end{equation*}
  where $G(s)$ is as in Theorem \eqref{theorem:CN}, and where $g$ is an associated modular form.
  Taking $A = 1$, $\delta = 1$, $w = 1 + \epsilon$, $Q(x) = O(1)$, $q = -\infty$ and applying Theorem~\ref{theorem:CN}, we have
  \begin{equation} \label{balancing}
    \sum_{n \leq X} A_f(n) = O\left(X^{\frac{1}{4}+(\frac{1}{2}+\epsilon)\eta}\right)+O\left(\sum_{X \leq n\leq X'} |A_f(n)| \right),
  \end{equation}
  where $X'=X+O(X^{\frac{1}{2}-\eta})$.
  The bounds $A_f(n) \ll n^{\alpha(f) + \epsilon}$ lead to the trivial bound
  \begin{equation*}
    \sum_{X \leq n\leq X'} |A_f(n)| \ll X^{\frac{1}{2} - \eta + \alpha(f) + \epsilon}.
  \end{equation*}
  Balancing $\eta$ in~\eqref{balancing} we get
  \begin{equation*}
    \sum_{n \leq X} A_f(n) \ll O\left(X^{\frac{1}{3} + \frac{\alpha(f)}{3} + \epsilon}\right).
  \end{equation*}
  Inserting the bounds from Theorem~\ref{theorem:individual_bounds} completes the proof.
\end{proof}

By  inputting the bounds in Corollaries~\ref{cor:adj},~\ref{cor:sym}, and~\ref{cor:sum} and Theorem~\ref{theorem:individual_bounds} into Theorem~\ref{theorem:main_sign_changes} we get the following theorem on generalized sign changes.

\begin{theorem} \label{thm:gl2intervals}
  Let $f$ be a modular cusp form of level $N$, possibly with nontrivial nebentypus $\chi$, that is either a holomorphic form as in~\eqref{eq:holomorphic_expansion} or a weight zero Maass form as in~\eqref{eq:maass_expansion}.
If there is an $n$ such that $\Re A_f(n) \neq 0$ (resp. $\Im A_f(n) \neq 0$) then $\Re A_f(n)$ (resp. $\Im A_f(n)$) changes sign at least once for $n \in (X, X + X^{\frac{3}{5}}]$ for $X \gg 1$.
\end{theorem}

\section{Individual Coefficients of $\GL(3)$ Cusp Forms}
\label{sec:individual_gl3}

We now consider a Maass cusp form $\phi$ on $\SL_3(\mathbb{Z})$ with Fourier coefficients $A(m,n)$, normalized so that $A(1,1) = 1$, and which is a simultaneous eigenfunction of the Hecke operators.
We can write $\phi$ as
\begin{equation*}
  \phi(z) = \sum_{\gamma \in U_2(\mathbb{Z}) \backslash \SL_2(\mathbb{Z})} \sum_{m_1, m_2} A(m_1, m_2) W (m_1, m_2, \gamma, z),
\end{equation*}
where $W$ is Jacquet's Whittaker function (as in~\cite[Chapter 6.2]{Goldfeld2006automorphic}).
We investigate sign changes in $\{A(1,n)\}_{n \geq 1}$ using Theorem~\ref{theorem:signchanges_analytic}.

Recall the standard Godement-Jacquet $L$-function
\begin{equation*}
  L(s,f) := \sum_{n \geq 1} \frac{A(1,n)}{n^s} = \prod_p (1 - \alpha_{1,p}p^{-s})^{-1}(1 - \alpha_{2,p}p^{-s})^{-1}(1 - \alpha_{3,p}p^{-s})^{-1}.
\end{equation*}
As shown in~\cite{KS}, we have $\lvert \alpha_{i,p}\rvert \leq p^{\frac{5}{14}}$, giving the bound $\lvert A(1,n) \rvert \ll n^{\frac{5}{14} + \epsilon}$ for any $\epsilon > 0$.
Note that a direct application of the functional equation of $L(s,f)$ leads to the trivial bound $L(it, f) \ll (1 + \lvert t \rvert)^{3/2}$.

We introduce and study the following two Dirichlet series,
\begin{align*}
  D_1(s) &:= \sum_{n \geq 1} \frac{A(1,n)\overline{A(1,n)}}{n^s} \\
  D_2(s) &:= \sum_{n \geq 1} \frac{A(1,n)^2}{n^s}.
\end{align*}
These are not natural $L$-functions, but we can understand the behaviour of $D_1$ and $D_2$ by relating them to the $\GL(3)$ Rankin-Selberg convolution $L$-functions $L(s, f\times \overline{f})$ and $L(s, f\times f)$.

\begin{lemma}
  With the notation above,
  \begin{align*}
    D_1(s) &= \sum_{n \geq 1} \frac{A(1,n)\overline{A(1,n)}}{n^s} = L(s,f \times f) T_1(s) \\
    D_2(s) &= \sum_{n \geq 1} \frac{A(1,n)^2}{n^s} = L(s, f \times \overline{f}) T_2(s),
  \end{align*}
  where $L(s, f\times f)$ denotes the Rankin-Selberg convolution $L$-function
  \begin{equation*}
    L(s, f\times f) := \zeta(3s) \sum_{m,n} \frac{A(m_1, m_2)\overline{A(m_1, m_2)}}{m_1^{2s} m_2^s},
  \end{equation*}
  and where
  \begin{align*}
    T_1(s) &= \prod_p (1 - b_p p^{-2s} + (2b_p-2) p^{-3s} -b_p p^{-4s} + p^{-6s}) \\
    T_2(s) &= \prod_p \left(1 - \sum_{i,j=1}^3 \alpha_{i,p}^{-1} \alpha_{j,p}^{-1}p^{-2s} + \frac{2b_p - 2}{p^{3s}} - \sum_{i,j=1}^3 \alpha_{i,p} \alpha_{j,p}p^{-4s} + p^{-6s}\right) \\
    b_p &= \sum_{i=1}^3 \sum_{j=1}^3 \frac{\alpha_{i,p}}{\alpha_{j,p}}.
  \end{align*}
  Further, $T_1$ and $T_2$ are holomorphic and absolutely convergent for $\Re s > \frac{6}{7}$.
\end{lemma}

\begin{proof}
  By writing the sums as products over primes, we need only consider the $p$-parts.
  Note that $\overline{A(1, p^k)} = \widetilde{A}(1, p^k)$, where $\widetilde{A}$ are the coefficients of $\widetilde{f}$, the dual Maass form to $f$, and which has Satake parameters $\alpha_{1,p}^{-1},\alpha_{2,p}^{-1}$, and $\alpha_{3,p}^{-1}$.
  The coefficient $A(1, p^k)$ can be written explicitly in terms of the Satake parameters~\cite[Chapter 7.4]{Goldfeld2006automorphic} as
  \begin{equation*}
    A(1,p^k) = \frac{\alpha_{1,p}^{k+2}(\alpha_{2,p} - \alpha_{3,p}) + \alpha_{2,p}^{k+2}(\alpha_{3,p} - \alpha_{1,p}) + \alpha_{3,p}^{k+2}(\alpha_{1,p} - \alpha_{2,p})}{(\alpha_{2,p} - \alpha_{1,p})(\alpha_{3,p} - \alpha_{2,p})(\alpha_{1,p} - \alpha_{3,p})}.
  \end{equation*}
Using that $\alpha_{1,p}\alpha_{2,p}\alpha_{3,p}=1$ and letting $\alpha_{i+3,p} = \alpha_{i,p}$ for $i=1,2,3$, we have that
  \begin{align*}
    \sum_{k \geq 0} &\frac{|A(1,p^k)|^2}{p^{ks}}  \\
    &= -\sum_{k \geq 0} \sum_{i,j=1}^3 \frac{\alpha_{i,p}^2 \alpha_{j,p}^{-2}(\alpha_{i+1,p}-\alpha_{i+2,p})(\alpha_{j+1,p}^{-1}-\alpha_{j+2,p}^{-1}) (\alpha_{i,p} \alpha_{j,p}^{-1})^k}{p^{ks} (\alpha_{2,p} - \alpha_{1,p})^2 (\alpha_{3,p} - \alpha_{2,p})^2(\alpha_{1,p} - \alpha_{3,p})^2 } \\
    &= -\sum_{i,j=1}^3 \frac{\alpha_{i,p}^2 \alpha_{j,p}^{-2}(\alpha_{i+1,p}-\alpha_{i+2,p})(\alpha_{j+1,p}^{-1}-\alpha_{j+2,p}^{-1}) }{(\alpha_{2,p} - \alpha_{1,p})^2 (\alpha_{3,p} - \alpha_{2,p})^2(\alpha_{1,p} - \alpha_{3,p})^2 (1 - \alpha_{i,p} \alpha_{j,p}^{-1} p^{-s})}.
  \end{align*}
  Using Sage~\cite{sage}, one can quickly verify that this can be rewritten as


  \begin{align*}
    \sum_{k \geq 0} \frac{|A(1,p^k)|^2}{p^{ks}}= &\frac{1 - b_p p^{-2s} + (2b_p-2) p^{-3s} - b_p p^{-4s} + p^{-6s}}{\prod_{i=1}^3 \prod_{j=1}^3 (1-\frac{\alpha_{i,p}}{\alpha_{j,p}} p^{-s})},
  \end{align*}
  with $b_p$ as claimed in the statement of the Lemma.

  This shows that $D_1 = L(s, f\times f) T_1(s)$.
  With respect to convergence, $T_1(s)$ converges if and only if
  \begin{equation}\label{eq:T1_convergence_sum}
    \sum_{p} \frac{b_p}{p^{2s}}
  \end{equation}
  converges.
  As $b_p \ll (p^{\frac{5}{14}})^2 = p^{\frac{5}{7}}$, we see~\eqref{eq:T1_convergence_sum} converges for $\Re s > \frac{6}{7}$ immediately.

  The work for $D_2$ is very similar, and follows by manipulating the Satake parameters (perhaps facilitated again by a CAS such as Sage).
\end{proof}

As $L(s,f \times f)$ has a pole at $s = 1$ with residue proportional to $\langle f, f\rangle$, the Petersson inner product, we see that $D_1$ has a pole at $s = 1$ with residue proportional to $\langle f, f\rangle$ (and similarly for $D_2$).
Applying the Phragm\'{e}n-Lindel\"{o}f principle to $L(s, f\times f)$ and $L(s, f\times \overline{f})$ gives that for $\Re s = \frac{6}{7} + \epsilon$,
\begin{equation*}
  D_1(s), D_2(s) \ll (1 + \lvert t \rvert)^{\frac{9}{14} + \epsilon}.
\end{equation*}

We may now apply Theorem~\ref{theorem:signchanges_analytic} with
\begin{align*}
  \alpha &= \frac{5}{14} + \epsilon, \sigma_0 = 0, \beta = \frac{3}{2} + \epsilon, \\
  \gamma_{1} &= 1, R_1 \sim \langle f, f\rangle, \sigma_1 = \frac{6}{7} + \epsilon, \eta_{1} = \frac{9}{14} + \epsilon, \\
  \gamma_{2} &= 1, R_2 \sim \langle f, \overline{f}\rangle, \sigma_2 = \frac{6}{7} + \epsilon, \eta_{2} = \frac{9}{14} + \epsilon.
\end{align*}
When the coefficients are not all real-valued, note that $\lvert R_1 \rvert > \lvert R_2 \rvert$.
We conclude the following result for sign changes.

\begin{theorem} \label{GL3}
  Let $\phi$ be a Maass Hecke cusp form on $\SL_3(\mathbb{Z})$ with Fourier coefficients $A(m,n)$.
  For any $\epsilon > 0$, $\{\Re A(1,n)\}_{n \in \mathbb{N}}$ has at least one sign change for $n \in (X, X + X^{\frac{21}{23} + \epsilon}]$ for all $X \gg 1$.

  Further, if there are coefficients with $\Im A(1,n) \neq 0$, then for any $\epsilon >0$ we also have that $\{ \Im A(1,n)\}_{n \in \mathbb{N}}$ has at least one sign change for $n \in (X, X+X^{\frac{21}{23} + \epsilon}]$ for all $X \gg 1$.
\end{theorem}

It is worth discussing the above theorem in light of a more recent result due to Meher and Murty in~\cite{MM2}.
\begin{theorem}[Meher and Murty (2016)]\label{mmty}
Let $L(s, \pi)$ be the automorphic $L$-function associated to an automorphic irreducible self-dual cuspidal representation $\pi$ of $\GL_m(\mathbb{A}_{\mathbb{Q}})$
with unitary central character satisfying the Ramanujan conjecture, and let $a_\pi(n)$ be its $n$\textsuperscript{th} coefficient.
If the sequence $\{a_\pi(n)\}_{n=1}^\infty$ is a sequence of real numbers, then for any real number $r$ satisfying $\frac{m^2-1}{m^2+1} < r < 1$
 the sequence $\{a_\pi(n)\}_{n=1}^\infty$ has at least one sign change for
 $n \in [X,X+X^r]$.
 Consequently, the number of sign changes of $a_\pi(n)$ for $n \leq x$ is $\gg X^{1-r}$ for sufficiently large $X$.
 \end{theorem}

If we remove the assumption that the coefficients, $a_\pi(n)$, which in our notation correspond to $A(1,n)$,  satisfy the Ramanujan conjecture, the techniques of Theorem~\ref{mmty} do not improve upon Theorem~\ref{GL3} using presently known bounds.
Indeed, taking the notation for $\alpha, \beta$ and $\gamma$ as in~\cite{MM} and Theorem~\ref{theorem:main_sign_changes}, we know that we can take $\alpha = \frac{5}{14}$ and $\beta=\frac{3}{5}$, as shown in~\cite{MM2}, and so $\alpha+\beta = \frac{67}{70} > \frac{21}{23}$.
Note that using the analytic properties as in Theorem~\ref{theorem:signchanges_analytic} gives an improved result when compared to the original Meher-Murty axiomatization.

As in Theorem~\ref{mmty}, one might consider generalizing Theorem~\ref{GL3} to $\GL(m)$.
The main obstacle is deducing the region of convergence for the generalized form of $T_1(s)$ and $T_2(s)$ for $m > 3$.
It is not clear how this would compare to Theorem~\ref{mmty} as $m$ becomes large without adding assumptions.
The Ramanujan-Petersson conjecture alone is not enough to match Theorem~\ref{mmty} using Theorem~\ref{theorem:signchanges_analytic}, but if we also assume that the generalized analogues of the Dirichlet series $L(s,f)$, $D_1(s)$ and $D_2(s)$ all satisfy a Generalized Lindel\"{o}f Hypothesis on the central line, then we could take the interval for sign changes to be $[x,x+x^{\frac{1}{2}+\epsilon}]$ for $m\geq 2$.

\section{Sums of Coefficients of $\GL(2)$ Cusp Forms}
\label{sec:sums_gl2}

We now consider sequences of possibly complex coefficients $\{S_f^\nu(n)\}_{n \in \mathbb{N}}$ where $f$ is a weight zero Maass form or a holomorphic cusp form of full or half integer weight $k$ on a congruence subgroup $\Gamma \subseteq \SL_2(\mathbb{Z})$, possibly with nontrivial nebentypus.
To apply Theorem~\ref{theorem:main_sign_changes}, we begin by applying Theorem~\ref{theorem:CN} to Dirichlet series associated to $S_f^\nu$.

Here we take coefficients $a_f(n)$ to be ``unnormalized'' in that $a_f(n)=A_f(n)n^{\kappa(f)}$ with
\begin{equation*}
  \kappa(f) = \begin{cases}
    \frac{k-1}{2} & \text{if } f \; \text{is a holomorphic cusp form}, \\
    0 & \text{if } f \; \text{is a Maass form}
  \end{cases}
\end{equation*}
and with $A_f(n)$ as defined in Corollary~\ref{cor:adj}.
We let the Dirichlet series $L^\nu(s, f)$ be the series associated to the $\nu$-normalized coefficients
\begin{equation*}
  L^\nu(s, f) = \sum_{n \geq 1} \frac{a_f(n)n^{-\nu}}{n^s} = \sum_{n \geq 1} \frac{a_f(n)}{n^{s + \nu}},
\end{equation*}
which satisfies a functional equation of the shape
\begin{equation}
  \Lambda^\nu(s) := \frac{q(f)^{s/2}}{\pi^s}\Gamma(s + \nu)L^\nu(s, f) = \varepsilon(f) \Lambda^\nu(2 \kappa(f) + 1 - 2\nu - s, \widetilde{f}),
\end{equation}
where $\widetilde{f}$ represents the dual of $f$.
Note that $L^\nu(s,f) = L(s + \nu - \kappa(f), f)$, where this latter $L$-function is the standard automorphic $L$-function associated to $f$, with coefficients $A_f(n)$.
We choose $\delta = 2\kappa(f) + 1 - 2\nu$, $A = 1$, $w = \kappa(f) + 1 - \nu, w' = 0$, and $\eta = \frac{1}{6} + \frac{2}{3}\alpha(f)$, where $\alpha(f)$ is as in Theorem~\ref{theorem:individual_bounds}.
Then~\eqref{eq:CN_L1} guarantees that
\begin{equation}\label{eq:Sf_alpha}
  S_f^\nu(X) \ll X^{\kappa(f) + \frac{1}{3} + \frac{1}{3}\alpha(f) + \epsilon - \nu},
\end{equation}
for $\nu < \kappa(f)+\frac{1}{2}$.
As $2w - \delta - \frac{1}{A} = 0$, we also have that
\begin{equation} \label{mamabound}
  \sum_{n \leq X} \lvert S_f^\nu(n) \rvert^2 = c X^{2\kappa(f) + \frac{3}{2} - 2\nu} + O(X^{2\kappa(f) + 1 - 2\nu}\log^2 X)
\end{equation}
again for $\nu < \kappa(f)+\frac{1}{2}$.
The constant $c_{f,f}$ can be computed (for instance, see~\cite{hkldw1} for $f$ a holomorphic cusp form) to be
\begin{equation*}
  c_{f,f}= \frac{1}{4 \pi^2 (2\kappa + \frac{3}{2} - 2\nu)}\sum_{n \geq 1} \frac{\lvert a_f(n) \rvert^2}{n^{2\kappa + \frac{3}{2} - 2\nu}}.
\end{equation*}

 The asymptotic \eqref{mamabound} it is enough to prove the following proposition
\begin{proposition} Let $S_k(\Gamma, \chi)$ denote the space of holomorphic cusp forms of weight $k$ on the congruence subgroup $\Gamma$ with nebentypus $\chi$.
  By an abuse of notation, we will take the case where $k=0$ to be weight zero Maass forms.
  If $f\in S_k(\Gamma_1,\chi_1)$ and $g\in S_k(\Gamma_2,\chi_2)$, where $\chi_1$ and $\chi_2$ are respectively trivial unless $\Gamma_1=\Gamma_0(n_1)$ and $\Gamma_2=\Gamma_0(n_2)$, again respectively, then
\begin{equation}\label{babybound1}
  \sum_{n \leq X} S_f^\nu(n)\overline{S_g^\nu(n)}  = c_{f,g} X^{2\kappa(f) + \frac{3}{2} - 2\nu} + O(X^{2\kappa(f) + 1 - 2\nu}\log^2 X)
  \end{equation}
  where
  \[
  c_{f,g} = \frac{1}{4 \pi^2 (2\kappa + \frac{3}{2} - 2\nu)}\sum_{n \geq 1} \frac{a_f(n)\overline{a_g(n)}}{n^{2\kappa + \frac{3}{2} - 2\nu}},
  \]
  and
  \begin{equation}\label{babybound2}
  \sum_{n \leq X} S_f^\nu(n)S_g^\nu(n)  = c_{f,\overline{g}} X^{2\kappa(f) + \frac{3}{2} - 2\nu} + O(X^{2\kappa(f) + 1 - 2\nu}\log^2 X)
  \end{equation}
    where
  \[
  c_{f,\overline{g}} = \frac{1}{4 \pi^2 (2\kappa + \frac{3}{2} - 2\nu)}\sum_{n \geq 1} \frac{a_f(n)a_g(n)}{n^{2\kappa + \frac{3}{2} - 2\nu}}.
  \]
  Here we take $\kappa(f)=\kappa(g) = \frac{k-1}{2}$ if $k\neq 0$ and $\kappa(f)=0$ otherwise.

\end{proposition}
\begin{proof}
For sufficiently large $N$, $\Gamma(N) \subseteq \Gamma_1 \cap \Gamma_2$, so we can take that $f,g \in S_k(\Gamma(N),\chi_0)$ where $\chi_0$ denotes the trivial character.
Thus $h_1, h_2 \in S_k(\Gamma(N),\chi_0)$ where $h_1=f+g$ and $h_2=f+ig$ and so we have that
\[
|S^\nu_{h_1}(n)|^2 = |S^\nu_f(n)|^2+|S^\nu_g(n)|^2 +2\Re\left( S^\nu_f(n)\overline{S^\nu_g(n)}\right)
\]
and
\[
 |S^\nu_{h_2}(n)|^2 = |S^\nu_f(n)|^2+|S^\nu_g(n)|^2 +2\Im\left( S^\nu_f(n)\overline{S^\nu_g(n)}\right).
\]
Thus
\[
S^\nu_f(n)\overline{S^\nu_g(n)} = \frac{1}{2} \left( |S^\nu_{h_1}(n)|^2+ i |S^\nu_{h_2}(n)|^2-(1+i)\left( |S^\nu_f(n)|^2+|S^\nu_g(n)|^2 \right) \right).
\]
The asymptotic of each of the summands above is given by \eqref{mamabound} and so \eqref{babybound1} follows.

Let $T_{-1}$ denote the reflection Hecke operator, so that $T_{-1}F(x+iy) = F(-x+iy)$.
To get \eqref{babybound2} we note that it is easy to see that if $g\in S_k(\Gamma(N),\chi_0)$ then $\overline{T_{-1}g} \in S_k(\Gamma(N),\chi_0)$ and $S_{\overline{T_{-1}g}}(n) = \overline{S_g(n)}$.
So \eqref{babybound2} is just a special case of \eqref{babybound1}.
\end{proof}

In Section~\ref{sec:proof_prop_Sf_L1}, we use similar methods as in~\cite{hkldw1} to prove the following proposition.

\begin{proposition}\label{prop:Sf_L1}
  Let $f$  and $S_f^\nu$ be defined as above.
  Then for $\nu > 0$ and $\epsilon > 0$, we have
  \begin{equation*}
    \sum_{n \leq X} S_f^\nu(n) \ll X^{\kappa(f) + \frac{5}{6} + \frac{1}{3}\alpha(f) + \epsilon - \nu} + X.
  \end{equation*}
\end{proposition}

By applying Theorem~\ref{theorem:main_sign_changes}, we can determine the following generalized sign changes.
\begin{theorem}
 Let $f$ be a weight $0$ Maass form or a holomorphic cusp form of full or half integer weight $k$ on a congruence subgroup $\Gamma \subseteq \SL_2(\mathbb{Z})$, possibly with nontrivial nebentypus if $\Gamma=\Gamma_0(n_0)$.
 Suppose that $0 \leq \nu < \kappa(f) + \frac{1}{6} - \frac{2 \alpha(f)}{3}$.
 If $\Re a_f(n) \neq 0$ (resp. $\Im a_f(n) \neq 0$) for at least one coefficient $a_f(n)$, then the sequence $\{\Re S_f(n) \}_{n \in \mathbb{N}}$ (resp. $\{\Im S_f(n) \}_{n \in \mathbb{N}}$) has at least one sign change for some $n \in (X, X + X^{r(\nu)}]$ for $X \gg 1$, where
  \begin{equation*}
    r(\nu) = \begin{cases}
      \frac{2}{3} + \frac{2\alpha(f)}{3} + \epsilon & \text{if } \nu < \kappa(f) + \frac{\alpha(f)}{3} - \frac{1}{6}, \\
      \frac{2}{3} + \frac{2\alpha(f)}{3} + \Delta + \epsilon & \text{if } \nu = \kappa(f) + \frac{\alpha(f)}{3} - \frac{1}{6} + \Delta, \; 0 \leq \Delta < \frac{1}{3} - \frac{2\alpha(f)}{3}.
    \end{cases}
  \end{equation*}
\end{theorem}

As an immediate corollary, we get another proof of generalized sign changes of the coefficients $\{a_f(n)\}_{n \in \mathbb{N}}$, although for a larger interval.

\begin{remark}
  In this theorem, $\nu$ can be taken to be slightly larger than $\kappa(f)$.
  For instance, in the case of full-integer weight holomorphic cusp forms, we can take $\nu = \frac{k-1}{2} + \frac{1}{6} - \epsilon$ for any $\epsilon > 0$, so that $S_f^\nu$ is a sum of \emph{overnormalized} coefficients.
  Regular sign changes of sums of overnormalized coefficients suggests a certain regularity of the sign changes of individual coefficients $a_f(n)$.
\end{remark}

\begin{remark}
  This proof and methodology applies also to any degree $2$ $L$-function with a functional equation and which satisfies the Ramanujan-Petersson conjecture on average.
\end{remark}

\section{Analytic Bounds}
\label{sec:analytic_bounds}

In this section, we prove bounds for the magnitude of $\sum_{n \leq X} S_f^\nu(n)$ necessary for use in generalized sign change theorems, as applied to the sequence $\{S_f^\nu(n)\}_{n \in \mathbb{N}}$.

For $X, Y > 0$, let $v_Y(X)$ denote a smooth non-negative function with maximum value $1$ satisfying
\begin{enumerate}
  \item $v_Y(X) = 1$ for $X \leq 1$,
  \item $v_Y(X) = 0$ for $X \geq 1 + \frac{1}{Y}$.
\end{enumerate}
Let $V(s)$ denote the Mellin transform of $v_Y(X)$, given by
\begin{equation*}
  V(s)=\int_0^\infty t^s v_Y(t) \frac{dt}{t}.
\end{equation*}
when $\Re(s) > 0$.
Through repeated applications of integration by parts, one can show that $V(s)$ satisfies the following properties:
\begin{enumerate}
  \item $V(s) = \frac{1}{s} + O_s(\frac{1}{Y})$.
  \item $V(s) = -\frac{1}{s}\int_1^{1 + \frac{1}{Y}}v'(t)t^s dt$.
  \item For all positive integers $m$, and with $s$ constrained to within a vertical strip where $|s-1|>\epsilon$, we have
  \begin{equation} \label{vbound}
  V(s) \ll_\epsilon \frac{1}{Y}\left(\frac{Y}{1 + \lvert s \rvert}\right)^m.
  \end{equation}
\end{enumerate}
Property $(3)$ above can be extended to real $m > 1$ through the Phragm\'en-Lindel\"{o}f principle.

\begin{proof}[Proof of Proposition~\ref{prop:Sf_L1}]
\label{sec:proof_prop_Sf_L1}

 We are led to investigate
\begin{equation*}
  L(s, S_f^\nu) := \sum_{n \geq 1}\frac{S_f^\nu(n)}{n^s}.
\end{equation*}
Similar to the methods in~\cite{hkldw1}, we decouple the variables through the use of a Mellin-Barnes integral,
\begin{equation}\label{mbarnes}
  \begin{split}
    L(s, S_f^\nu) 
    &= \sum_{n \geq 1}\sum_{h \geq 0} \frac{a_f(n)}{n^\nu (n+h)^s} \\
    &= L^\nu(s,f) + \frac{1}{2\pi i} \int_{(\gamma)} L^\nu(s -w,f) \zeta(w)\frac{\Gamma(w)\Gamma(s - w)}{\Gamma(s)}dw,
  \end{split}
\end{equation}
where $\Re(s)>\gamma>1$.

On the one hand, supposing $X,Y \gg 1$,
\begin{align*}
  \frac{1}{2\pi i} \int_{(\sigma)} L(s, S_f^\nu) V(s) X^s ds &= \sum_{n \geq 1} S_f^\nu(n) v(\tfrac{n}{X}) \\
  &= \sum_{n \leq X} S_f^\nu(n) + \sum_{X < n \leq X + X/Y} S_f^\nu(n) v(\tfrac{n}{X}) \\
  &= \sum_{n \leq X} S_f^\nu(n) + O\bigg( \frac{X}{Y} X^{\kappa(f) + \frac{1}{3} + \frac{1}{3}\alpha(f) + \epsilon - \nu}\bigg),
\end{align*}
where in the last bound we have used~\eqref{eq:Sf_alpha}.

On the other hand, from~\eqref{mbarnes} we have that
\begin{equation*}
  \begin{split}
    \frac{1}{2\pi i} &\int_{(\sigma)} L(s, S_f^\nu) V(s) X^s ds = \\
    &= \frac{1}{2\pi i} \int_{(\sigma)} L^\nu(s) V(s) X^s ds \\
    &\quad + \left(\frac{1}{2\pi i}\right)^2 \iint\limits_{(\sigma)(\gamma)} L^\nu(s-w) \zeta(w) \frac{\Gamma(w) \Gamma(s - w)}{\Gamma(s)} V(s) X^s dw ds
  \end{split}
\end{equation*}
for $\sigma > \gamma > 1$.
We bound the contribution from the first term by
\begin{equation*}
  \begin{split}
    \frac{1}{2\pi i} \int_{(\sigma)} L^\nu(s)V(s)X^sds &= \sum_{n\geq 1} \frac{a_f(n)}{n^\nu}v\left(\tfrac{n}{X} \right) \\
    & \ll X^{\kappa(f)+\frac{1}{3} + \frac{1}{3}\alpha(f) + \epsilon -\nu}+ O \left(\frac{X}{Y}X^{\kappa(f) + \alpha(f) -\nu} \right).
  \end{split}
\end{equation*}

For the double integral term, we perform a change of variables $s \mapsto s + w$ and shift the $s$ line of integration to $(\epsilon)$ for some small $\epsilon > 0$.
Then we have
\begin{align*}
  & \left(\frac{1}{2\pi i}\right)^2 \iint\limits_{(\sigma)(\gamma)}L^\nu(s-w)\zeta(w)\frac{\Gamma(w) \Gamma(s - w)}{\Gamma(s)}V(s) X^s dw ds \\
  &=
  \left(\frac{1}{2\pi i}\right)^2 \iint\limits_{(\epsilon)(\gamma)}L^\nu(s)\zeta(w)\frac{\Gamma(w) \Gamma(s )}{\Gamma(s+w)}V(s+w) X^{s+w} dw ds.
\end{align*}
Shifting the $w$ line of integration left past the pole at $w = 1$ to $(\epsilon)$ gives that this becomes
\begin{align}
  & \left(\frac{1}{2\pi i}\right) \int\limits_{(\epsilon)}L^\nu(s)\frac{V(1+s)}{s} X^{1+s}ds \label{line:res1}\\
  &\quad+ \left(\frac{1}{2\pi i}\right)^2 \iint\limits_{(\epsilon)(\epsilon)}L^\nu(s)\zeta(w)\frac{\Gamma(w) \Gamma(s )}{\Gamma(s+w)}V(s+w) X^{s+w} dw ds. \label{line:rem1}
\end{align}

We first consider the residual term~\eqref{line:res1}.
Shifting the line of integration for $s$ left past the pole at $s = 0$ gives the residue $L^\nu(0) V(1) X = O(X)$, and there are no additional poles.
By Stirling's approximation and the functional equation of $L(s,f)$, we have that ${L(-\sigma+it,f)} \ll (1+|t|)^{1+2\sigma}$ for $\sigma >0$, and so $L^\nu(-\sigma+it) \ll (1+|t|)^{2\kappa(f) + \frac{3}{2}-2\nu+2\sigma}$ for sufficiently large $\sigma>0$.
Recall that \eqref{vbound} holds for any $m \geq 1$.
Then by shifting the line of integration for $s$ to $(-\sigma)$ so that $\sigma \gg 1$, and taking $m = 2\kappa(f) + \frac{3}{2}-2\nu+2\sigma+\epsilon$, we have that~\eqref{line:res1} converges absolutely and is bounded by\begin{equation*}
  O(X^{1 - \sigma}Y^{2\sigma+2\kappa(f) + \frac{1}{2}-2\nu+\epsilon}).
\end{equation*}
If $Y < X^{\frac{1}{2} - \epsilon}$, then the remaining integral term can be made arbitrarily small as $X \to \infty$ by taking $\sigma$ to be sufficiently large.
Then~\eqref{line:res1} can be bounded by the pole at $s = 0$, which has size $O(X)$.

We now return our attention to~\eqref{line:rem1}.
Shifting $s$ to the left picks up poles due to the $\Gamma(s)$ term in the integrand.
The pole at $s = 0$ has residue
\begin{equation*}
  \frac{1}{2\pi i} \int_{(\epsilon)} L^\nu(0)\zeta(w) V(w) X^w \ dw.
\end{equation*}
For fixed $\sigma > 0$, we have that $\zeta(-\sigma + iu) \ll (1+|u|)^{\frac{1}{2}+\sigma}$.
Shifting the $w$ line of integration to $(-\sigma)$ with $\sigma \gg 1$ with $m = \sigma + 2$, we see that this shifted integral converges absolutely and is bounded by $O(X^{-\sigma} Y^{\sigma+1})$, which can be made arbitrarily small provided $Y < X^{\frac{1}{2} - \epsilon}$.
The pole at $w = 0$ due to $V(w)$ contributes a term bounded by $O(1)$ independent of $X$ and $Y$.

The poles at $s = -n$ for $n \in \mathbb{Z}_{> 0}$ have residues bounded by the residue at $s = 0$.
In particular, the residue at $s = -n$ is
\begin{equation*}
  \frac{1}{2\pi i} \int_{(\epsilon)} L^\nu(-n)\zeta(w) \frac{V(w-n)}{(w+n-1)\cdots(w+1)(w)} X^{w-n} \ dw \ll X^{\epsilon-n}.
\end{equation*}
We note that the implicit constant depends on $n$, and might grow extremely large in $n$.
For this reason all of our shifts will be independent of $X$, $Y$.
For now, we shift our line of integration in $s$ to $(-M)$ for some $M > 0$ to be specified later.

Letting $s = \sigma + it$ and $w = \epsilon + iu$, we consider the shifted integral using Stirling's approximation and see that for sufficiently negative $\sigma$, \begin{align*}
  & \left( \frac{1}{2\pi i} \right)^2 \iint\limits_{(\sigma)(\epsilon)} L^\nu(s) \zeta(w) \frac{\Gamma(w) \Gamma(s )}{\Gamma(s+w)} V(s+w) X^{s+w} dw ds  \\
  & \ll \int\limits_{-\infty}^\infty \int\limits_{-\infty}^\infty  \lvert L(\sigma+\nu-\kappa(f)+it,f) \zeta(\epsilon+iu)\rvert \frac{(1+\lvert u \rvert )^{\epsilon-\frac{1}{2}}(1+ \lvert t \rvert )^{\sigma-\frac{1}{2}}}{(1+ \lvert t+u \rvert )^{\sigma+\epsilon-\frac{1}{2}}} \\
  & \quad \times e^{-\frac{\pi}{2}( \lvert t \rvert + \lvert u \rvert - \lvert t+u \rvert )} \lvert V(s+w) \rvert  X^{\sigma+\epsilon} dudt \\
  & \ll \!\int\limits_{-\infty}^\infty \!\int\limits_{-\infty}^\infty \frac{(1+ \lvert t \rvert )^{2\kappa(f) + 1-2\nu-\sigma}(1+ \lvert u \rvert )^{\frac{\epsilon}{2}}}{(1+ \lvert t+u \rvert )^{\sigma+\epsilon-\frac{1}{2}+m}e^{\frac{\pi}{2}( \lvert t \rvert + \lvert u \rvert - \lvert t+u \rvert )}} Y^{m-1}X^{\sigma+\epsilon} du dt.
\end{align*}
The largest contribution from this integral occurs when $t$ and $u$ have the same sign, when there is no exponential damping.
For these $t$ and $u$, $(1 + \lvert t+u \rvert)$ is at least as large as $(1 + \lvert t \rvert)$ and $(1+\lvert u \rvert)$ and it is clear that the integral converges for sufficiently large $m$.
Taking $\sigma$  to $-M$ and choosing $m = 2\kappa(f) - 2\nu + 2M + 2$, we bound this contribution absolutely by
\begin{equation}\label{eq:Sf_doubleintegral_bound}
  O(X^{-M + \epsilon} Y^{2\kappa(f) - 2\nu + 2M + 2}).
\end{equation}
If $Y \ll X^{\frac{1}{2} - \epsilon}$, we may choose $M$ large enough so that~\eqref{eq:Sf_doubleintegral_bound} is smaller than $O(X)$.

Combining the error terms, we have
\begin{equation*}
  \sum_{n \leq X} S_f^\nu(n) = O\bigg(\frac{X}{Y}X^{\kappa(f)+\frac{1}{3} + \frac{1}{3} \alpha(f) + \epsilon - \nu} \bigg)+O(X).
\end{equation*}
Choosing $Y = X^{\frac{1}{2} - \epsilon}$ gives the proposition.
\end{proof}

\vspace{20 mm}
\bibliographystyle{elsarticle-harv}
\bibliography{signchangebib}

\end{document}